\documentclass[11pt]{amsart}
\usepackage{paper}
\usepackage{float}
\usepackage{longtable}

\renewcommand{\C}{\mathcal{C}}

\title{Bounding Lifts of Markoff Triples $\mod p$}

\author{Elisa Bellah}
\address{(Bellah) Department of Mathematics, Carnegie Mellon University}
\author{Siran Chen}
\address{(Chen) Department of Mathematics, Carnegie Mellon University}
\author{Elena Fuchs}
\address{(Fuchs) Department of Mathematics, University of California, Davis}
\author{Lynnelle Ye}
\begin{document}

\begin{abstract}
 In 2016, Bourgain, Gamburd, and Sarnak proved that Strong Approximation holds for the Markoff surface in most cases. That is, the modulo $p$ solutions to the equation $X_1^2+X_2^2+X_3^2=3X_1X_2X_3$ are covered  by the integer solutions for most primes $p$. In this paper, we provide upper bounds on lifts of $\mod p$ points of the Markoff surface by analyzing the growth along paths in the Markoff $\mod p$ graphs. Our first upper bound follows the algorithm given in the paper of Bourgain, Gamburd, and Sarnak, which constructs a path of possibly long length but where points grow relatively slowly. Our second bound considers paths in these graphs of short length but possibly large growth. We then provide numerical evidence and heuristic arguments for how these bounds might be improved. 
 \end{abstract}

	\maketitle

	\section{Introduction}
	\label{sec: 1}

The Markoff surface $\X$ is the affine surface in $\A^3$ given by
\[\X: X_1^2+X_2^2+X_3^2=3X_1X_2X_3,\]
and the integer points $\X(\Z)$ are called Markoff Triples.  Markoff first studied these triples in the context of Diophantine approximation and the theory of quadratic forms (see \cite{markoff} or Chapter 2 of \cite{cassels}, for example). Markoff triples have since found application in other areas. For example, Cohn used these triples to study free groups on two generators (see Theorem 2 of \cite{cohn}), and in \cite{Hirzebruch} Hirzebruch and Zagier used Markoff triples to study the signature of certain 4-dimensional manifolds. More recently, work has been done to study Markoff surface modulo $p$. \\

Define the Vieta group $\Gamma$ to be the group of affine integral morphisms on $\A^3$ generated by permutations $\sigma_{ij}$ and the Vieta involutions $R_i$; that is
\[\sigma_{12} (x_1, x_2, x_3) = (x_2, x_1, x_3),\]
\[ R_1(x_1, x_2, x_3) = (3 x_2 x_3 - x_1, x_2, x_3),\]
and the remaining $\sigma_{ij}$ and $R_2, R_3$ are defined similarly. It is well-known that the orbit of $(1, 1, 1)$ under $\Gamma$ gives all the nonzero Markoff triples (see \cite{aigner}, for example). It is expected that the same holds modulo $p$. More precisely, we have the following conjecture. 

\begin{conjecture}[Strong Approximation] For any prime $p$, 
	\[\Gamma \cdot (1, 1, 1) = \X(\Z/p\Z) - \{(0, 0, 0)\}.\]
\end{conjecture}

In Theorem 1 of \cite{BGS}, Bourgain, Gamburd and Sarnak show that Strong Approximation holds for primes $p$ where $p^2-1$ does not have too many divisors. The authors prove this theorem by showing algorithmically that the Markoff mod $p$ graphs, defined in Section \ref{sec: 2}, are connected.\\

It is conjectured that the Markoff $\mod p$ graphs in fact form an expander family, and so in light of \cite{lauter} these graphs have been proposed as a means to produce cryptographic hash functions. In \cite{elena}, it is noted that one avenue of attack depends on the difficulty of finding lifts of Markoff triples modulo $p$. \\

Throughout this paper, we assume that $p>3$ is prime. For convenience, we set the following notation and terminology. We denote the nonzero $\mod p$ points $\X(\Z/p\Z)-\{(0, 0, 0)\}$ by $\X^*(p)$ and call elements in $\x \in \X^*(p)$ Markoff $\mod p$ points. We refer to the algorithm given by Bourgain, Gamburd, and Sarnak in \cite{BGS} to BGS algorithm. \\

Define the size of a Markoff triple $\x=(x_1, x_2, x_3)$ by
\[\size(\x):=\max\{x_1, x_2, x_3\}.\] 
Furthermore, for a Markoff $\mod p$ point $\x \in \X^*(p)$ we call a Markoff triple $\tilde{\x}$ a lift of $\x$ if we have $\tilde{\x} \equiv \x (\mod p)$. The main results of this paper give bounds on the size of minimal lifts by considering ``optimal" paths in the Markoff $\mod p$ graphs, which are defined in Section \ref{sec: 2} along with relevant definitions. We show the following.

\begin{theorem} \label{thm1} Let $p$ be a prime so that $\ord_p(\rot_1^n(1, 1, 1))\geq p-1$ for $0 \leq n \leq 5$ and suppose that $\x \in \X^*(p)$ with $\ord_p(\x)>p^{1/2}$. Let $\tilde{\x}$ be a lift of $\x$ of minimal size. Then,
	\[\size(\tilde{\x}) < (3\epsilon)^{96(2p+1)^4},\]
	where $\epsilon=(3+\sqrt{5})/{2}$.
\end{theorem}

In Section \ref{sec:thm1}, we give numerical evidence to show that the conditions in Theorem \ref{thm1} appear to be satisfied for many primes $p$ and approximately 80\% of Markoff $\mod p$ points. Furthermore, we provide direction for how these results might be obtained theoretically, and for obtaining upper bounds on the sizes of smallest lifts for the remaining points $\x$ in $\X^*(p)$ with $\ord_p(\x)\leq p^{1/2}$.\\

Our second upper bound is conditional on Strong Approximation, and depends on the expansion constant $h(p)$ of the Markoff $\mod p$ graph $\G_p$, but holds for all points $\x$ in $\X^*(p)$. It is expected that the family of Markoff $\mod p$ graphs $(\G_p)_p$ forms an expander family, and so this bound is expected to be uniform (see the discussion on Super Strong Approximation in Section \ref{sec:thm2}). We show the following.

\begin{theorem} \label{thm2} 
	Let $p$ be any prime where Strong Approximation holds, and let $h(p)$ be the expansion constant of the Markoff graph $\G_p$. For $\x \in \X^*(p)$, let $\tilde{\x}$ be a lift of $\x$ of minimal size. Then,
	\[\size(\tilde{\x}) < (3 \epsilon)^\alpha,\]
	where $\epsilon=(3+\sqrt{5})/{2}$ and 
	\[\alpha = \left(\frac{p^3+3}{2}\right)^{20/\log\left(1+\frac{h(p)}{3}\right)}.\]
\end{theorem}

\vspace{0.5em}
This paper is organized as follows. In Section \ref{sec: 2} we introduce special elements of the Vieta group, called rotations, and give an explicit upper bound for the growth of Markoff triples under the action of the group of rotations. In Sections \ref{sec: 2} and \ref{sec:BGSAlg} we give the necessary background and an explicit description of the algorithm given in the paper \cite{BGS} of Bourgain, Gamburd, and Sarnak for which Theorem \ref{thm1} rests, and provide several alternate proofs to those given in \cite{BGS} using the language of linear recurrence sequences. In Section \ref{sec:thm1} we prove Theorem \ref{thm1} and give numerical evidence and a heuristic argument for how one might relax the conditions of this Theorem. In Section \ref{sec:thm2} we prove Theorem \ref{thm2} and give evidence for how this bound might be improved on average. Finally, in Section \ref{sec:parabolic} we discuss an alternate approach to finding paths in the Markoff $\mod p$ graphs which could be used for further improvements to the bounds given in Theorems \ref{thm1} and \ref{thm2}.\\

\section{Rotations and The Markoff $\mod p$ Graphs} \label{sec: 2}

To analyze lifts, we consider the following special elements in the Vieta group $\Gamma$ as introduced in \cite{BGS}. 

\begin{definition} \label{rotdef} The \textit{rotations} of $\Gamma$ are the elements $\rot_i$ given by
	\[\rot_1 = \sigma_{23} R_2, \rot_2 = \sigma_{13} R_1, \text{ and } \rot_3 = \sigma_{12} R_1.\]
Explicitly, we have
\[\rot_1(x_1, x_2, x_3) = (x_1, x_3, 3x_1x_3-x_2)\]
\[\rot_2(x_1, x_2, x_3) = (x_3, x_2, 3x_2x_3 - x_1)\]
\[\rot_3(x_1, x_2, x_3) = (x_2, 3x_2x_3-x_1, x_3).\]
For a prime $p$, the \textit{Markoff $\mod p$ graph} $\G_p$ is defined to be the graph with vertex set $\X^*(p)$ and edges $(\x, \rot_i \x)$ for $i \in \{1, 2, 3\}$. 
	\end{definition}

In Figure \ref{fig:graph} we give an example of the Markoff mod $p$ graph $\G_p$ when $p=31$ to demonstrate the complexity of the path structure for these graphs. In this figure, nodes are colored depending on their rotation order, with blue nodes corresponding to points in the cage (defined in Section \ref{sec:cage}).

\begin{figure}
    \includegraphics[scale=0.35]{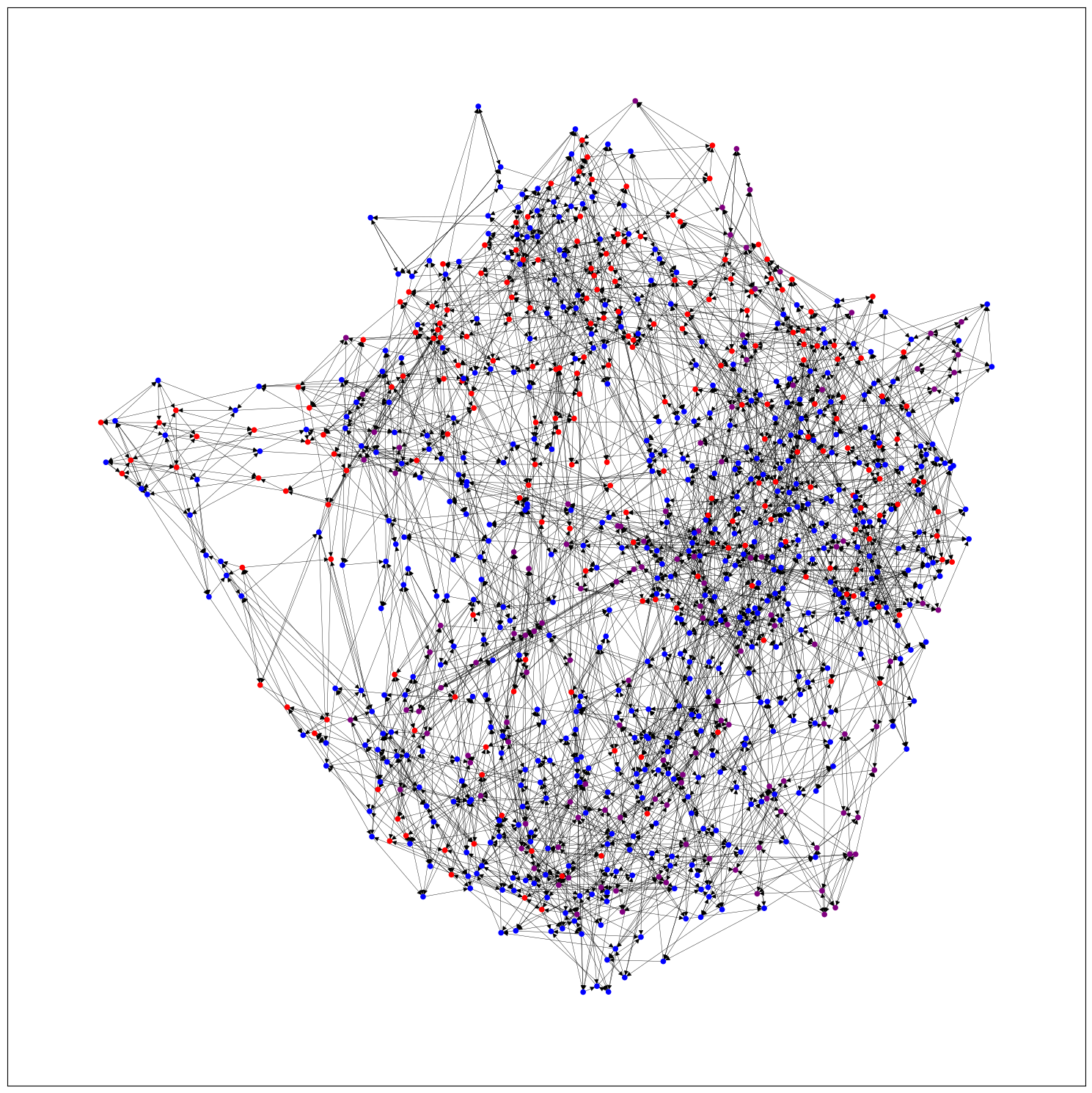}
    \caption{The Markoff $\mod p$ graph $\G_p$ when $p=31$}
    \label{fig:graph}
\end{figure}

\subsection{Sizes under action of rotations}  It's known that in most cases, the Markoff $\mod p$ graph $\G_p$ is connected (see \cite{BGS} and \cite{chen}). Observe that when $\G_p$ is connected, we can construct lifts of points in $\X^*(p)$ by finding paths between $(1, 1, 1)$ and points $\x \in \G_p$. Proposition \ref{sizes} will tell us that finding small lifts modulo $p$ can then be done by finding ``optimal" paths in $\G_p$. Observe that we have the following.

\begin{lemma} \label{LRS} Given $(x_1, x_2, x_3) \in \X^*(p)$ we have
	\[\rot_1^n(x_1, x_2, x_3) = (x_1, a_n, a_{n+1}),\]
	where $a_n$ is the linear recurrence sequence with initial conditions $a_0=x_2, a_1=x_3$ and recurrence
	\[a_{n+2} = 3x_i a_{n+1} - a_n.\]
	Similarly, $\rot_i^n(x_1, x_2, x_3) = \sigma (x_i, a_n, a_{n+1})$ where the sequence $a_n$ has initial conditions given by the other two coordinates $x_i, x_j$ and recurrence as above, and $\sigma$ is a suitable permutation.
\end{lemma}

This observation allows us to analyze the growth of points obtained from $(1, 1, 1)$ by the action of $\langle \rot_1, \rot_2, \rot_3 \rangle$. We have the following.

\begin{proposition} \label{sizes} Let $n_i \in \Z_{\geq 1}$ and $i_j \in \{1, 2, 3\}$. Then we have
	\[\size(\rot_{i_s}^{n_s} \cdots \rot_{i_1}^{n_1}(1, 1, 1)) \leq (3\epsilon)^{2^{s-1}(n_1+1) \cdots (n_s+1)},
	\]
	where $\epsilon = \frac{3+\sqrt{5}}{2}$. 
\end{proposition}


Note that an exponential lower bound can be derived using a similar method to that outlined below, and so we obtain a result similar to that of Zagier in \cite{zagier}, which bounds the growth of sizes in the Markoff tree. Furthermore, we see that switching between rotations contributes doubly exponentially in growth, while traveling along a single rotation only contributes exponentially. Theorem \ref{thm1} follows the path constructed by Bourgain, Gamburd and Sarnak in \cite{BGS} which has a relatively small number of switches between different rotations, but possibly long path lengths along each orbit. Theorem \ref{thm2} instead considers paths of shortest possible length, but possibly many switches between distinct rotations.

\begin{proof}[Proof of Proposition \ref{sizes}] Let $\x=(x_1, x_2, x_3)$ be a Markoff triple, and without loss of generality set $x=x_1$ and suppose that $x>0$. Let $a_n$ denote the linear recurrence sequence defined in Lemma \ref{LRS}, and $\epsilon_x, \bar{\epsilon}_x$ denote the characteristic roots of $a_n$. That is, $\epsilon_x, \bar{\epsilon}_x$ are the roots of the minimal polynomial of the sequence $a_n$, which is given by
	\[f_x(T)=T^2-3xT+1.\]
Note that $f_x$ has distinct positive real roots, so we can set $\epsilon_x > \bar{\epsilon}_x$. Since $a_n$ is a linear recurrence sequence we can write
\[a_n=A \epsilon_x^n +B\bar{\epsilon}_x^n\]
for some $A, B \in \Q(\epsilon)$. Using our initial conditions we have
\[x_2 = A+B \text{ and } x_3= A\epsilon_x +B\bar{\epsilon}_x,\]
and so solving for $A$ and $B$ gives
\[
	a_{n+1} =\frac{(x_3-x_2\bar{\epsilon}_x) \epsilon_x^{n+1} + (\epsilon_x x_2-x_3)\bar{\epsilon}_x^{n+1}}{\epsilon_x-\bar{\epsilon}_x}\]
\begin{equation}
	\label{eq0}
	=x_3 \cdot \frac{\epsilon_x^{n+1}-\bar{\epsilon}_x^{n+1}}{\epsilon_x-\bar{\epsilon}_x}
	- x_2 \cdot \frac{\epsilon_x^{n}-\bar{\epsilon}_x^{n}}{\epsilon_x-\bar{\epsilon}_x},
\end{equation}
where the second equality uses that $\epsilon_x\bar{\epsilon}_x=1$. \\


	We induct on $s \in \Z_{\geq 1}$. When $s=1$, Equation (\ref{eq0}) gives

\begin{align*}
\size(\rot_{i_1}^{n_1}(1, 1, 1)) & < \frac{\epsilon^{n+1}-\bar{\epsilon}^{n+1}}{\epsilon-\bar{\epsilon}} \\
&< \frac{\epsilon^{n+1}}{\epsilon-\bar{\epsilon}}\\ &<\epsilon^{n+1},
\end{align*}
where $\epsilon=(3+\sqrt{5})/2$, as desired. \\

Next, suppose that
\[\x=\rot_{i_{s-1}}^{n_{s-1}} \cdots \rot_{i_1}^{n_1}(1, 1, 1)\]
and let $x$ denote the $i_s$th coordinate of $\x$. We have
\begin{align*}
	\size(\rot_{i_s}^{n_s}(\x)) & < \size(\x) \, \epsilon_x^{n_s+1}, \text{ by Equation (\ref{eq0})} \\
								& = \size(\x) \cdot \left(\frac{3 x + \sqrt{(3x)^2-4}}{2}\right)^{n_s+1}\\
								& < \size(\x) \cdot \left(\frac{3 x + \sqrt{(3x)^2}}{2}\right)^{n_s+1}\\
								& = \size(\x) \cdot (3 \size(\x))^{n_s+1},
\end{align*}
and so the result follows by induction. \end{proof}

\begin{remark} 
In Figure \ref{fig:growth}, the sizes of points \[\rot_{i_s}\cdots \rot_{i_1}(1, 1, 1)\] are graphed with the number of rotations $s$ on the horizontal axis. 
Observe that our upper bound in Proposition \ref{sizes} overshoots the growth of these sizes. However, through our experimentation is appears that a tight upper bound will still depend doubly exponentially on $s$ and exponentially on the $n_i$. Improvements to this upper bound or derivation of such an asymptotic would improve the results of this paper, but we expect the true upper bound to still depend on a balance of finding paths in $\G_p$ which are short (i.e. minimizes the $n_i$, which we expect contributes exponentially to the asymptotic growth as in Proposition \ref{sizes}) but switch between distinct rotations minimally (i.e. minimizes $s$, which we contribute doubly exponentially to the asmptotic growth as in Proposition \ref{sizes}).

\begin{figure}
    \includegraphics[scale=0.8]{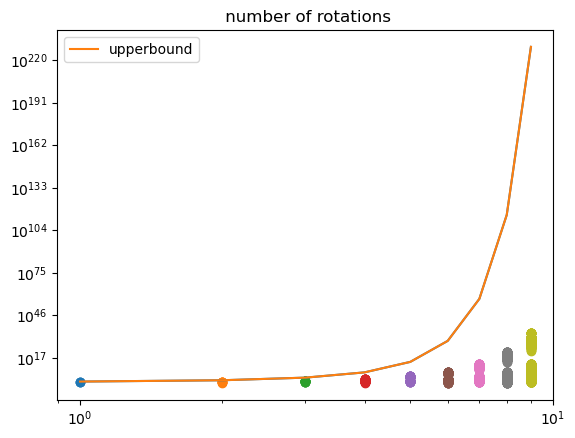}
    \caption{Size of $\rot_{i_s}\cdots 
    \rot_{i_1} (1, 1, 1)$}
    \label{fig:growth}
\end{figure}

\end{remark}

\subsection{Rotation Orders and Conic Sections} \label{sec:rotords}
The BGS algorithm constructs paths in the Markoff mod $p$ graph by analyzing the orbits under the rotations, defined in \ref{rotdef}. We present the analysis of these orbits as in \cite{BGS}, giving alternate proofs of several results. 

\begin{definition} For $\x = (x_1, x_2, x_3) \in \X^*(p)$, we define the following.
	\begin{enumerate}
		\item The \textit{$i$th rotation order} of $\x$ is given by
		\[\ord_{p, i}(\x) := \min \{n \in \Z_{>0} \mid \rot_i^n (X) \equiv X (\mod p)\}.\]
		\item The \textit{rotation order of $\x$} is given by
		\[ \ord_p(\x) : = \max\{ \ord_{p, i}(x) \mid i=1, 2, 3\}.\]
	\end{enumerate}
If $\ord_{p}(\x)=\ord_{p, i}(\x)$, then we call $i$ the \textit{maximal index} of $\x$. 
\end{definition}

Observe that for distinct $i, j, k$, $\rot_i$ acts on $(x_j, x_k)$, and we can write
\[\rot_i(x_i) \begin{pmatrix} x_j \\ x_k \end{pmatrix} = \begin{pmatrix} 0  & 1 \\ -1 & 3x_i \end{pmatrix} \begin{pmatrix} x_k \\ x_j \end{pmatrix}.\]
So, the $i$th rotation order of $\x$ is equal to the order of
\[ \begin{pmatrix} 0 & 1 \\ -1 & 3x_i \end{pmatrix}\]
in $\SL_2(\F_p)$. Note in particular that $\rot_{p, i}(\x)$ only depends on the $i$th coordinate of $\x$. \\

\begin{definition}\label{notation_orbit}
	We set the following notation. For $\x=(x_1, x_2, x_3) \in \X^*(p)$, let $M_{\x, i}:=\{\rot_i^n (\x) \mid n \in \Z\}$
be the orbit of $\x$ under $\rot_i$. Furthermore, let
\[A_{\x, i}:=\begin{pmatrix} 0 & 1 \\ -1 & 3x_i \end{pmatrix}\]
be the matrix discussed above, and denote the characteristic polynomial of $A_{\x, i}$ by $f_{\x, i}$ and the discriminant of $f_{\x, i}$ by $\Delta_{\x, i}$. In our analysis below, we will sometimes only be concerned with a single coordinate $x$ of a Markoff $\mod p$ point. In this case, we will instead use the notation  $M_x, A_{x}, f_x$ and $\Delta_{x}$, respectively. 
\end{definition}

Note that the action of $\rot_i$ on a Markoff triple $\x$ leaves the $i$th coordinate of $\x$ fixed, and so the orbits $M_{\x, i}$ correspond to points inside of some conic section with discriminant $\Delta_{\x, i}$. Given $a \in \Z/p\Z$ we use the notation
\[C_i(a):=\{(x_1, x_2, x_3) \in \X^*(p) \mid x_i = a\}\]
to denote the conic section consisting of Markoff $\mod p$ points with $i$th coordinate equal to $a$. Note that we may also use the notation $C_i(\x)$ to indicate the conic section cut out by fixing the $i$th coordinate of $\x$. Accordingly, we have the following definition. 

\begin{definition}\label{sectionsdef} Let $\x=(x_1, x_2, x_3) \in \X^*(p)$ be a Markoff $\mod p$ point with $i$th coordinate equal to $x$. We say that $x$ is
	\begin{enumerate}
	\item \textit{parabolic} if $\Delta_{x} \equiv 0 (\mod p)$,  
	\item \textit{hyperbolic} if $\Delta_x$ is a nonzero square modulo $p$, and
	\item \textit{elliptic} if $\Delta_x$ is not a square modulo $p$.
	\end{enumerate} 
\end{definition}

We have the following observation.

\begin{lemma} \label{ordepsilon} Let $\x=(x_1, x_2, x_3)$ be a Markoff $\mod p$, and set $x=x_i$. Let $\epsilon_x$ a root of $f_x$. If $x$ is not parabolic, then $\ord_{p, i}(\x)$ is equal to the order of $\epsilon_x$ in $\F^\times$, where 
	\[\F= \begin{cases} \F_p & \text{ when $x$ is hyperbolic} \\
		\F_{p^2}& \text{ when $x$ is elliptic} \end{cases}\] 
under the identification $\F_{p^2} \cong \F_p[T]/(T^2-\Delta_x).$
\end{lemma}

\begin{proof} When $x$ is hyperbolic or elliptic, $A_x$ is diagonalizable over $\F$ with eigenvalues given by $\epsilon_x$ and its conjugate $\bar{\epsilon}_x$. Since $\epsilon_x, \bar{\epsilon}_x$ are roots of \[f_x(T)=T^2-3xT+1\] then we have $\bar{\epsilon}_x=\epsilon_x^{-1}$. This gives
	\[A_x^n = P^n \begin{pmatrix} \epsilon_x^n & 0 \\ 0 & \epsilon_x^{-n} \end{pmatrix} P^{-n}\]
where $P \in \SL_n(\F)$. So, the order of $A_x$ is equal to the order of $\epsilon$ in $\F$.
\end{proof}

The following Proposition can be found in Lemma 3 of \cite{BGS}. We present an alternate proof of this result here.

\begin{proposition} \label{ord} Let $\x=(x_1, x_2, x_3) \in \X^*(p)$. For a prime $p >3$. We have 
	\[\ord_{p, i}(\x) \text{ divides }
	\begin{cases} p-1 & \text{ if $x_i$ is hyperbolic} 
		\\ p+1 & \text{ if $x_i$ is elliptic}.
	\end{cases}
	\]
If $x_i$ is parabolic, then $x_i=\pm 2/3$ and we have
\[\ord_{p, i}(\x) = \begin{cases} 2p & \text{ if } x_i = -2/3 \\ p & \text{ if } x_i=2/3. \end{cases}\]
\end{proposition}

\begin{proof} For convenience, set $x=x_i$ and $\epsilon=\epsilon_x$. If $x$ is hyperbolic, then the result follows directly from Lemma \ref{ordepsilon}. So suppose that $x_i$ is elliptic. Since 
	\[\epsilon+\bar{\epsilon} =3x_i \in \F_p^\times\]
	then we get
	\[(\epsilon+\bar{\epsilon})^p = \epsilon+\bar{\epsilon},\]
	and since $\bar{\epsilon}=\epsilon^{-1}$ then
	\[\epsilon^p+\epsilon^{-p} = \epsilon+\epsilon^{-1}.\]
	So, by Lemma \ref{alglem} below, we have $\epsilon= \epsilon^p$ or $\epsilon=\epsilon^{-p}$. But if $\epsilon=\epsilon^p$ then $\epsilon^{p-1}=1$ which implies $\epsilon\in \F_p^\times$, contradicting $x_i$ being elliptic. So we must have $\epsilon=\epsilon^{-p}$ which gives $\epsilon^{p+1}=1$. Hence, the order of $\epsilon$ divides $p+1$ and the result follows from Lemma \ref{ordepsilon}.\\
	
	Finally, suppose that $x$ is parabolic. Then $\Delta_x \equiv 0 (\mod p)$ which gives
	\[(3x)^2-4 \equiv 0 (\mod p) \Rightarrow x \equiv \pm 2/3 (\mod p).\]
	Note that in the parabolic case $A_x$ has Jordan normal form
	\[A_x = P \begin{pmatrix} \epsilon & 1 \\ 0 & \epsilon \end{pmatrix} P^{-1},\]
	for some $P \in \GL_2(\Z/p\Z)$ and so
	\begin{equation} \label{ax} A_x^n = P \begin{pmatrix} \epsilon^n & n \epsilon \\ 0 & \epsilon^n \end{pmatrix} P^{-1}.\end{equation}
	Recalling that $\epsilon_x$ denotes a root of the characteristic polynomial of $A_x$, we can compute
	\[\epsilon_x = \begin{cases} 1 & \text{ if } x = 2/3 \\ -1 & \text{ if } x =-2/3. \end{cases}.\]
	So, the result for the order of parabolic points follows from Equation (\ref{ax}) and Lemma \ref{ordepsilon}.
	\end{proof}

\begin{lemma} \label{alglem} If $a+a^{-1}=b+b^{-1}$ for nonzero elements $a, b$ in a field $\F$ with $\text{char} (\F) \not=2$ then we have $a=b$ or $a=b^{-1}$. 
	\end{lemma}

\begin{proof} If $a+a^{-1}=b+b^{-1}$ then multiplying both sides by $ab$ gives
	\[a^2b+b=ab^2+a \Rightarrow a^2b-a(b^2+1)+b=0.\]
Since $\text{char}(\F) \not=2$ the quadratic formula gives the desired result.
\end{proof}

\begin{definition} \label{maxord} Let $\x\in \X^*(p)$. 
	\begin{enumerate}
		\item If $\ord_{p, i}(\x)=p-1$ then we say $x_i$ is \textit{maximal hyperbolic}. 
		\item If $\ord_{p, i}(\x)=p+1$ then we'll call $x_i$ is \textit{maximal elliptic}, and 
		\item If $\ord_{p, i}(\x)=2p$ we say $x_i$ is \textit{maximal parabolic}. 
	\end{enumerate} 
A triple $\x \in \X^*(p)$ will be called \textit{maximal} (hyperbolic, elliptic, or parabolic) if one of its coordinates is either maximal hyperbolic, elliptic, or parabolic. 
\end{definition}

\subsection{Connection to Lucas Sequences} In Equation (\ref{eq0}), we saw that the orbit of points $\x$ in $\X^*(p)$ under a single rotation can be described in terms of a corresponding Lucas sequence. We prove the following identity giving this relation explicitly for any Markoff $\mod p$ points, which may be useful for future study, and is referenced in Section \ref{sec:111} as a possible direction to relax the conditions of Theorem \ref{thm1}. 

\begin{lemma} \label{matrixidentity} Let $\x \in \X^*(p)$ have $i$th coordinate $x$. Then, 
\[
 A_x^n = \begin{pmatrix} -u_{n-1} & u_n \\ -u_n & u_{n+1} \end{pmatrix},
\]
where $A_x$ is the matrix defined in \ref{notation_orbit} and $u_n$ is the Lucas sequence with integer parameters $(3x, 1)$. That is, $u_n$ is the linear recurrence sequence with initial conditions $u_0=0, u_1=1$ and recurrence $u_{n+2}=3x u_{n+1}-u_n.$
\end{lemma}

\begin{proof} For convenience, set $x=x_i$. Observe that
\[ A_x=\begin{pmatrix} 0 & 1 \\ -1 & 3x \end{pmatrix} = \begin{pmatrix} 3x & -1 \\ 1 & 0 \end{pmatrix} ^{-1}.\]
The matrix on the right-hand side gives a familiar Lucas sequence identity. We have
\begin{align*}
\begin{pmatrix} 0 & 1 \\ -1 & 3x \end{pmatrix} ^n	& = \begin{pmatrix} u_{n+1} & -u_n \\ u_n & -u_{n-1} \end{pmatrix}^{-1} \\ 
												& = \frac{1}{u_n^2 - u_{n+1}u_{n-1}} \begin{pmatrix} -u_{n-1} & u_n \\ - u_n & u_{n+1} \end{pmatrix}. 
\end{align*} 
Now, let $\epsilon, \bar{\epsilon}$ be the roots of 
\[f(T) = T^2 - 3xT+1.\]
Then, we can write
\[u_n = \frac{\epsilon^n-\bar{\epsilon}^n}{\epsilon-\bar{\epsilon}}\]
and we note that $\epsilon\bar{\epsilon} =1$. So we have
\begin{align*}
u_n^2-u_{n+1}u_{n-1} 	& = \left(\frac{\epsilon^n-\bar{\epsilon}^n}{\epsilon-\bar{\epsilon}}\right)^2 - \left(\frac{\epsilon^{n+1}-\bar{\epsilon}^{n+1}}{\epsilon-\bar{\epsilon}}\right)\left(\frac{\epsilon^{n-1}-\bar{\epsilon}^{n-1}}{\epsilon-\bar{\epsilon}}\right)\\
						& = \frac{1}{(\epsilon-\bar{\epsilon})^2} \left((\epsilon^{2n}+\bar{\epsilon}^{2n} -2) - (\epsilon^{2n}+\bar{\epsilon}^{2n}-(\epsilon^2+\bar{\epsilon}^2))\right)\\
						& = 1,
\end{align*}
as required.
\end{proof}

\section{The BGS Algorithm} \label{sec:BGSAlg}

The algorithm outlined in \cite{BGS} constructs a path between any two points in the Markoff $\mod p$ graph by connecting them through the cage, which we discuss below. The BGS algorithm then guarantees a path between any two points in the cage which contains at most two distinct rotations. In light of Proposition \ref{sizes}, the paths constructed in the BGS algorithm are a good candidate for small lifts, particularly when one or both of our points are in the cage. In this section, we give an outline of the key components of the BGS algorithm from \cite{BGS}, extracting explicit information when possible. An explicit description of the path constructed by the BGS algorithm is then outlined in Algorithm \ref{BGSalg}.

\subsection{The Cage}\label{sec:cage}

Define the \textit{cage} to be the subgraph $\C(p)$ of the Markoff $\mod p$ graph $\G_p$ containing all vertices that are maximal points in $\X^*(p)$, as defined in \ref{maxord}. The convenience of the cage is that the orbits with respect to the maximal index are precisely equal to the conic sections at that index. More precisely, we have the following.

\begin{proposition} \label{conicorbit}
If $\x=(x_1, x_2, x_3)$ is in the cage $\C(p)$ with maximal index $i$ then $M_{\x, i}=C_i(x_i)$. That is, if $\x$ is a maximal triple with maximal index $i$, then the orbit of $\x$ under $\rot_i$ contains all Markoff $\mod p$ points with $i$th coordinate equal to $x_i$.  
\end{proposition}

This Proposition follows from our definition of maximal points along with the following lemma which can be found in \cite{BGS}.

\begin{lemma} \label{conic} If $x$ is parabolic, then 
	\[ |C_i(x)| = \begin{cases} 0 & \text{ if } p \equiv -1 (\mod 4) \\ 2p & \text{ if } p \equiv 1 (\mod 4) \end{cases}.\]
	Otherwise, we have
	\[|C_i(x)| = \begin{cases} p+1 & \text{ if $x$ is elliptic} \\ p-1 & \text{ if $x$ is hyperbolic}. \end{cases}\]
	\end{lemma}

The connectedness of the cage will then follow from the following result. For completeness of our analysis and implementation of the BGS algorithm, we outline the proof given in \cite{BGS}.

\begin{lemma}[Section 3.2 of \cite{BGS}] \label{cage} Let $\x, \y \in \C(p)$ with maximal indices $i, j \in \{1, 2, 3\}$, respectively. Then there is a point $\z \in \C(p)$ and $k \in \{1, 2, 3\}$ so that 
	\[C_i(\x) \cap C_k(\z) \not=\varnothing \text{ and } C_j(\y) \cap C_k(\z) \not=\varnothing.\]
\end{lemma}

\begin{proof}
	Let $\x=(x_1, x_2, x_3)$ and $\y=(y_1, y_2, y_3)$. Suppose first that the maximal indices of $\x$ and $\y$ are distinct. Without loss of generality, say $\x$ has maximal index $i=1$ and $\y$ has maximal index $j=2$. We claim that there are elements $\zeta, \gamma, z$ in $\F_p$ so that
	\[(x_1, \zeta, z), (\gamma, y_2, z)\in \X^*(p).\]
	Considering the Markoff equation as a quadratic in $\zeta$, for $(x_1, \zeta, z)$ to be in $\X^*(p)$ we must have
	\[(3x_1z)^2-4(x_1^2+z^2)= \alpha^2\]
	for some $\alpha \in \F_p$. Similarly, considering the Markoff equation as a quadratic in $\gamma$, for $(\gamma, y_2, z)$ to be in $\X^*(p)$ we must have
	\[(3y_2 z)^2-4(y_2^2+z^2) = \beta^2\]
	for some $\beta \in \F_p$. Rearranging, this gives the system of equations
	\[(9x_1^2-4)z^2-\alpha^2=4x_1^2\]
	\[(9y_2^2-4)z^2-\beta^2=4y_2^2\]
	with unknowns $z, \alpha, \beta$. \\
	
	When $x_1^2 \not=y_2^2$, this system of equations defines an irreducible curve in $\A^3$, which has solutions modulo any prime. 
	If $x_1^2=y_2^2$ then this reduces to finding solutions $z, \alpha$ to
	\begin{equation} \label{above} (9x_1^2-4)z^2-\alpha^2=4x_1^2. \end{equation}
	If $x_1$ is parabolic, then $9x_1^2-4 \equiv 0 (\mod p)$. Since we've assumed that $C_1(\x)$ is nonempty, then by Lemma \ref{conic} we must have $p \equiv 1 (\mod 4)$ an so $-1$ is a square modulo $p$, which gives a solution to equation (\ref{above}). If $x_1$ is not parabolic, then equation (\ref{above}) is a conic section, which has a solution modulo any prime. Using an inclusion/exclusion argument, the authors of \cite{BGS} show that such a solution can be found so that $z$ has maximal order. Hence, if we let $\z=(x_1, \zeta, z)$ then $\z \in \C(p)$ and $C_3(\z)$ intersects $C_1(\x)$ and $C_2(\y)$ nontrivially. Note that if $\x$ and $\y$ have the same maximal index, then we can find points $(x_1, \zeta, z), (y_1, \gamma, z) \in \X^*(p)$ by solving the Equation (\ref{above}) as above. This gives triple $\z=(x_1, \zeta, z) \in \X^*(p)$ with $z$ maximal and $(y_1, \gamma, z) \in C_3(\z)$. 
\end{proof}

\subsection{Connecting points to the cage} The following two Propositions from \cite{BGS} describe how to connect points of large enough order to the cage. We give a brief outline of the proofs given in \cite{BGS}, both to gather explicit information and to highlight the nonexplicit steps in this construction which require we assume long path lengths in the proof of Theorem \ref{thm1}. Note that these Propositions do not cover the case when $\x=(x_1, x_2, x_3)$ has maximal index $i$ and $x_i$ is parabolic of order $p$. We will instead deal with this separately in Proposition \ref{parabolic}.

\begin{proposition}[The Endgame from Section 4 of \cite{BGS}] \label{endgame} Let $\x$ be in $\X^*(p)$ with $\ord_{p}(\x)>p^{1/2}$, and suppose that $\x=(x_1, x_2, x_3)$ has maximal index $i$ with $x_i$ not parabolic. Then, there exists a positive integer $n$ and $i \in \{1, 2, 3\}$ so that $\rot_i^n(\x)$ is a maximal triple. 
\end{proposition}

\begin{proof}
	 Let $\F$ denote $\F_p$ when $x_i$ is hyperbolic and $\F_{p^2}$ if $x_i$ is elliptic, under the identification $\F_{p^2} \cong \F_p[T]/(T^2-\Delta_{x_i})$, as in the proof of Lemma \ref{ordepsilon}.
	Without loss of generality, let $i=1$. By Lemma \ref{LRS} we can write 
	\[\rot_1^n(\x) = (x_1, C_1 \epsilon^n+C_2 \bar{\epsilon}^n, C_1 \epsilon^{n+1}+C_2 \bar{\epsilon}^{n+1}),\]
	where $\epsilon, \bar{\epsilon}$ are roots of $f(T)=T^2-3x_iT+1$ and $C_1, C_2$ are constants in $\F$.  The proof given in \cite{BGS} uses the Weil bound along with an inclusion/exclusion argument to show that there exists a positive integer solution $n$ to
	\[C_1\epsilon^n+C_2 \bar{\epsilon}^n = \rho + \rho^{-1}\]
	where $\rho$ is defined as follows: in the hyperbolic case, $\rho$ is any generator of $\F_p^\times$, and in the elliptic case we take $\rho = u^{p-1}$ where $u$ generates $\F_{p^2}^\times$. Then $\rot_i^n(\x)$ has second coordinate equal to $\rho+\rho^{-1}$, and since $\rho, \rho^{-1}$ are the roots of \[f(T)=T^2-(\rho+\rho^{-1})T+1\] then by Lemma \ref{ordepsilon} $\ord_{p, 2}(\y)$ is equal to the order of $\rho$, which is maximal by construction. 
\end{proof}

\begin{proposition}[The Middlegame from Section 4 of \cite{BGS}] \label{middlegame} Let $\x$ be in $\X^*(p)$ with 
$c < \ord_p(\x)\leq p^{1/2},$
	where $c$ is a fixed constant independent of $p$, which is described in \cite{BGS}. Then, there exist $i_1, \dots, i_t \in \{1, 2, 3\}$ and positive integers $m_1, \dots, m_t$ so that the rotation order of 
	\[\rot_{i_t}^{m_t} \cdots \rot_{i_1}^{m_1} (\x)\] is larger than $p^{1/2}$, where $t \leq \tau(p^2-1)$. 
\end{proposition}

\begin{proof} Since the rotation order of a parabolic triple is lower bounded by $p$, we know that $\x$ is hyperbolic or elliptic. So, by Proposition \ref{ord}, we have $\ord_{p}(X) \mid p^2-1$. Suppose that $\x$ has maximal index $i$ and let $M_{\x, i}$ be the orbit of $\x$ under $\rot_i$, as defined in \ref{notation_orbit}. Let $\y=(y_1, y_2, y_3)$ be in $M_{\x, i}$ with maximal index $j$. If $y_j$ is parabolic, then $\ord_p(\y) \geq p$ and we're done. Otherwise, $y_j$ is hyperbolic or elliptic, and so $\ord_p(\y)$ divides $p^2-1$ as well. If $\ord_p(\y)>\ord_p(\x)$ then $i\not=j$ and we repeat the process above by considering the orbit $M_{\y, j}$. If $\ord_p(\y) \leq \ord_p(\x)$, then we consider the sum
	\[N_\x := \sum_{\ell \leq |M_{\x, i}|} \#\{\y \in M_{\x, i} : |M_{\y, j}| = \ell'\}.
	\]
	The authors of \cite{BGS} use estimates on the gcd of elements of the form $u-1, v-1$ from Corvaja and Zannier in \cite{gcd} to show that if $p^2-1$ does not have too many prime divisors, then $N_\x < |M_{\x, i}|$. So, by definition of $N_\x$ there must be a point in the orbit of $\x$ with larger rotation order, and we proceed as above. That is, given our point $\x$, we can find $\z_1 \in M_{\x, i}$ with $\ord_p(\z_1)>\ord_p(\z)$. If $\z_1$ has maximal coordinate $z_1$ parabolic, then we're done. Otherwise, $\ord_p(\y) \mid p^2-1$. Iterating in this way, we obtain $\z_1, \z_2, \dots \in \X^*(p)$ with
	\[\ord_p(\x) < \ord_p(\z_1) < \ord_p(\z_2) < \cdots \]
	and since $\ord_p(\z_i) \mid (p^2-1)$ for every $i$, this process will stop after at most $\tau(p^2-1)$ steps, where $\tau(n)$ denotes the number of positive divisors of an integer $n$.
\end{proof}

\subsection{Connecting Parabolic Points to the Cage} \label{parabolic} Suppose  $\x=(x_1, x_2, x_3)$ is in $\X^*(p)$ with $x_i$ parabolic. For convenience, suppose that $i=1$ and set $x=x_1$. From Proposition \ref{ord} and Lemma \ref{conic}, we know that $p \equiv 1 (\mod 4)$ and if $x=2/3$ then $C_1(x)$ consists of two disjoint orbits. Here, we describe the conic section in terms of these orbits, and use this to show how to connect parabolic points of order $p$ to any point in the cage.  \\

The Jordan normal decomposition of $A_{2/3}$ is given by
\[A_{2/3}=\begin{pmatrix} 1 & -1 \\ 1 & 0 \end{pmatrix} \begin{pmatrix} 1 & 1 \\ 0 & 1 \end{pmatrix} \begin{pmatrix} 0 & 1 \\ -1 & 1 \end{pmatrix}\]
and so we get

\[A_{2/3}^n = \begin{pmatrix} 1 & -1 \\ 1 & 0 \end{pmatrix} \begin{pmatrix} 1 & n \\ 0 & 1 \end{pmatrix} \begin{pmatrix} 0 & 1 \\ -1 & 1 \end{pmatrix}.\]

Let 
\[\x^{(+)}=\left(\frac{2}{3}, 1, 1+\frac{2}{3}i\right) \text{ and } 
\x^{(-)}=\left(\frac{2}{3}, 1, 1-\frac{2}{3} i \right)\]
where $i$ is the square root of $-1$ modulo $p$, noting again that the assumption of a parabolic point implies that $p \equiv 1 (\mod 4)$. From above we get
\begin{equation} \label{eqx-}
	M_{\x^{(+)}}=\left\{\left(\frac{2}{3}, 1+\frac{2}{3} in , 1+\frac{2}{3} i (n+1)\right) \mid n \in \Z \right\}
\end{equation}
\begin{equation} \label{eqx+}
	M_{\x^{(-)}}=\left\{\left(\frac{2}{3}, 1-\frac{2}{3} in , 1-\frac{2}{3} i (n+1)\right) \mid n \in \Z \right\}.
\end{equation}
Observe that these orbits are disjoint and each contain $p$ elements, so we have
\[C_i(\x)=M_{\x^{(+)}} \cup M_{\x^{(-)}}.\]
Now, let $\y=(y_1, y_2, y_3)$ be any point in the cage with maximal index $j \in \{2, 3\}$ with and $y_j \not=0$. Since $M_{\x^{(+)}}$ contains $p-1$ points with distinct $j$th coordinates and $y_j \in \F_p^\times$ then we must have $C_j(y) \cap M_{\x^{(+)}} \not=\varnothing$ and similarly we have $C_j(y) \cap  M_{\x^{(-)}} \not=\varnothing$. Hence, we can connect $\x$ to the cage by $\rot_1$.

\subsection{Points of small order}

 The authors of \cite{BGS} show that points $\x$ with order bounded from above by a constant not depending on $p$ are connected to the cage. This proof is nonconstructive, and so we do not include such points in Theorem \ref{thm1}. We refer the reader to Section 5 of \cite{BGS} on ``The Opening" for this analysis.

\subsection{Connecting $(1, 1, 1)$ to the cage} \label{sec:111}

In the tables below, we demonstrate that $\rot_1^n(1, 1, 1)$ is in the cage for all primes $p \leq 199$ where $0 \leq n \leq 5$. If we can show that $(1, 1, 1)$ is always ``close" to the cage, or find a large family of primes where this holds, we could then relax the conditions of Theorem \ref{thm1}. One direction to do this might be to note the following connection to the Fibonacci sequence. Observe that the Lucas sequence $u_n$ defined in Lemma \ref{matrixidentity} when $\x=(1, 1, 1)$ is equal to $f_{2n}$. This gives
\[A_1^n = \begin{pmatrix} -f_{2n-2} & f_{2n} \\ - f_{2n} & f_{2n+2}\end{pmatrix}\]
and so 
\begin{align*}
    \rot_1^n(1, 1, 1) &=(1, f_{2n}-f_{2n-2}, f_{2n+2}-f_{2n})\\
                    & = (1, f_{2n-1}, f_{2n+1}).
\end{align*}
So, one approach would be to investigate the orders of $A_x$ where $x$ is an odd term in the Fibonacci sequence. 

\renewcommand{\arraystretch}{1.2}
\begin{table}[H]
\begin{tabular}{|c|c|c|}
\hline
\,\, Prime \,\, & \,\,\, Path \,\,\, & Points \\ \hline
2 & $\rot_1^1$ & $(1, 1, 1)$, $(1, 1, 0)$ \\\hline
3 & $\rot_1^1$ & $(1, 1, 1)$, $(1, 1, 2)$ \\\hline
5 & $\rot_1^1$ & $(1, 1, 1)$, $(1, 1, 2)$ \\\hline
7 & $\rot_1^1$ & $(1, 1, 1)$, $(1, 1, 2)$ \\\hline
11 & $\rot_1^1$ & $(1, 1, 1)$, $(1, 1, 2)$ \\\hline
13 & $\rot_1^1$ & $(1, 1, 1)$, $(1, 1, 2)$ \\\hline
17 & $\rot_1^1$ & $(1, 1, 1)$, $(1, 1, 2)$ \\\hline
19 & $\rot_1^1$ & $(1, 1, 1)$, $(1, 1, 2)$ \\\hline
23 & $\rot_1^1$ & $(1, 1, 1)$, $(1, 1, 2)$ \\\hline
29 & $\rot_1^2$ & $(1, 1, 1)$, $(1, 1, 2)$, $(1, 2, 5)$ \\\hline
31 & $\rot_1^2$ & $(1, 1, 1)$, $(1, 1, 2)$, $(1, 2, 5)$ \\\hline
37 & $\rot_1^1$ & $(1, 1, 1)$, $(1, 1, 2)$ \\\hline
41 & $\rot_1^2$ & $(1, 1, 1)$, $(1, 1, 2)$, $(1, 2, 5)$ \\\hline
43 & $\rot_1^1$ & $(1, 1, 1)$, $(1, 1, 2)$ \\\hline
47 & $\rot_1^3$ & $(1, 1, 1)$, $(1, 1, 2)$, $(1, 2, 5)$, $(1, 5, 13)$\\\hline
53 & $\rot_1^1$ & $(1, 1, 1)$, $(1, 1, 2)$ \\\hline
59 & $\rot_1^5$ & $(1, 1, 1)$, $(1, 1, 2)$, $(1, 2, 5)$, $(1, 5, 13)$, $(1, 13, 34)$, $(1, 34, 30)$\\\hline
61 & $\rot_1^1$ & $(1, 1, 1)$, $(1, 1, 2)$ \\\hline
67 & $\rot_1^1$ & $(1, 1, 1)$, $(1, 1, 2)$ \\\hline
71 & $\rot_1^2$ & $(1, 1, 1)$, $(1, 1, 2)$, $(1, 2, 5)$ \\\hline
\end{tabular}
\label{tab:111} 
\caption{Paths from $(1, 1, 1)$ to $\C(p)$}
\end{table}
\begin{table}
\begin{tabular}{|c|c|c|}
\hline
\,\, Prime \,\, & \,\,\, Path \,\,\, & Points \\ \hline
73 & $\rot_1^1$ & $(1, 1, 1)$, $(1, 1, 2)$ \\\hline
79 & $\rot_1^2$ & $(1, 1, 1)$, $(1, 1, 2)$, $(1, 2, 5)$ \\\hline
83 & $\rot_1^1$ & $(1, 1, 1)$, $(1, 1, 2)$ \\\hline
89 & $\rot_1^3$ & $(1, 1, 1)$, $(1, 1, 2)$, $(1, 2, 5)$, $(1, 5, 13)$\\\hline
97 & $\rot_1^1$ & $(1, 1, 1)$, $(1, 1, 2)$ \\\hline
101 & $\rot_1^1$ & $(1, 1, 1)$, $(1, 1, 2)$ \\\hline
103 & $\rot_1^1$ & $(1, 1, 1)$, $(1, 1, 2)$ \\\hline
107 & $\rot_1^1$ & $(1, 1, 1)$, $(1, 1, 2)$ \\\hline
109 & $\rot_1^1$ & $(1, 1, 1)$, $(1, 1, 2)$ \\\hline
113 & $\rot_1^5$ & $(1, 1, 1)$, $(1, 1, 2)$, $(1, 2, 5)$, $(1, 5, 13)$, $(1, 13, 34)$, $(1, 34, 89)$\\\hline
127 & $\rot_1^1$ & $(1, 1, 1)$, $(1, 1, 2)$ \\\hline
131 & $\rot_1^1$ & $(1, 1, 1)$, $(1, 1, 2)$ \\\hline
137 & $\rot_1^1$ & $(1, 1, 1)$, $(1, 1, 2)$ \\\hline
139 & $\rot_1^1$ & $(1, 1, 1)$, $(1, 1, 2)$ \\\hline
149 & $\rot_1^1$ & $(1, 1, 1)$, $(1, 1, 2)$ \\\hline
151 & $\rot_1^3$ & $(1, 1, 1)$, $(1, 1, 2)$, $(1, 2, 5)$, $(1, 5, 13)$\\\hline
157 & $\rot_1^1$ & $(1, 1, 1)$, $(1, 1, 2)$ \\\hline
163 & $\rot_1^1$ & $(1, 1, 1)$, $(1, 1, 2)$ \\\hline
167 & $\rot_1^1$ & $(1, 1, 1)$, $(1, 1, 2)$ \\\hline
173 & $\rot_1^1$ & $(1, 1, 1)$, $(1, 1, 2)$ \\\hline
179 & $\rot_1^3$ & $(1, 1, 1)$, $(1, 1, 2)$, $(1, 2, 5)$, $(1, 5, 13)$\\\hline
181 & $\rot_1^1$ & $(1, 1, 1)$, $(1, 1, 2)$ \\\hline
191 & $\rot_1^3$ & $(1, 1, 1)$, $(1, 1, 2)$, $(1, 2, 5)$, $(1, 5, 13)$\\\hline
193 & $\rot_1^1$ & $(1, 1, 1)$, $(1, 1, 2)$ \\\hline
197 & $\rot_1^1$ & $(1, 1, 1)$, $(1, 1, 2)$ \\\hline
199 & $\rot_1^1$ & $(1, 1, 1)$, $(1, 1, 2)$ \\\hline
\end{tabular}
\label{tab:1112} \caption{Paths from $(1, 1, 1)$ to $\C(p)$ continued}
\end{table}

\newpage
\subsection{Path Construction} We are now prepared to give an explicit description of the path from $(1, 1, 1)$ to points $\x \in \G_p$ constructed in the BGS algorithm.  

\begin{algorithm} \label{BGSalg} Let $p$ be prime and let $\x \in \X^*(p)$ with $\ord_p(X)>c$, where $c$ is defined in the Middlegame of \cite{BGS}. Then, we can connect $(1, 1, 1)$ to $X$ in $\G_p$ as follows. 
	\begin{enumerate}
		\item[I.] We assume that $\rot_1^n(1, 1, 1)$ is connected to the cage for $0 \leq n \leq 5$, which appears numerically to hold for many primes $p$ as discussed in Section \ref{sec:111}. 
		
		\item[II.] If $\x \in \C(p)$ then we can connect $\y$ to $\x$ as follows. Suppose that $\x$ has maximal index $i_1$ and $\y$ has maximal index $i_2$. By Lemma \ref{cage} there exists a point $\z \in \C(p)$ with maximal index $i_1 \in \{1, 2, 3\}$ so that
		\[C_{i_3}(\x) \cap C_{i_2}(\z) \not=\varnothing \text{ and } C_{i_1}(\y) \cap C_{i_2}(\z) \not=\varnothing.\]
		Furthermore, by Proposition \ref{conicorbit} we have that
		\[C_{i_3}(\x)=M_{i_3}(\x), C_{i_2}(\z)=M_{i_2}(\z) \text{ and } C_{i_1}(\y)=M_{i_1}(\y).\]
		So, the orbit of $\z$ under $\rot_{i_2}$ intersects the orbits of $\x$ under $\rot_{i_3}$ and the orbit of $\y$ under $\rot_{i_1}$. That is, there are points $\z_x, \z_y \in \C(p)$ with 
		\[\x= \rot_{i_3}^{n_3}(\z_x) \text{ and } \z_y = \rot_{i_1}^{n_1}(\y)\]
		for integers $n_1, n_3$. As in Figure \ref{fig: im}, this gives the following path from $(1, 1, 1)$ to our point $\x \in \C(p)$
		\[\x = \rot_{i_3}^{n_3} \rot_{i_2}^{n_2} \rot_{i_1}^{n_1} \rot_j^n  (1, 1, 1).\]

		\begin{figure} [h!] 
			\begin{center}
				\begin{tikzpicture}[scale=1.25]
					\draw [black,fill] (3, 0) circle [radius=0.05] node [black,below=.2] {$\x$};
					\draw [thick] (3,0) to [out=180,in=10] (0, 1) node[above right]{\small $\rot_{i_3}$}; 
					\draw [thick] (3, 0) to [out=180,in=10] (0.75, 0) node[below right] {\small $\rot_{i_1}$}; 
					\draw [thick] (3, 0) to [out=180,in=10] (0.5, -1) node[below right] {\small $\rot_{i_2}$}; 
					\draw [black,fill] (-3, 0) circle [radius=0.05] node [black,below=.2] {$\y$};
					\draw [thick] (-3,0) to [out=10,in=180] (-1.4, 1.5) node[left=.4]{\small $\rot_{i_3}$}; 
					\draw [thick] (-3,0) to [out=10,in=180] (-.4, 0) node[above left]{\small $\rot_{i_1}$}; 
					\draw [thick] (-3,0) to [out=10,in=180] (-1.2, -1) node[below left]{\small $\rot_{i_2}$}; 
					\draw [red,fill] (0, 2.5) circle [radius=0.05] node [red,above=.2] {$\z$};
					\draw [red,fill] (1.25, 0.75) circle [radius=0.05] node [red, right=.2] {\small $\z_x$};
					\draw [red,fill] (-1.45, 0.1) circle [radius=0.05] node [red, above left] {\small $\z_y$};
					\draw [red, thick] (0,2.5) to [out=5,in=100]  (1.25, 0.75) node [above=1.15]{\small $\rot_{i_2}$}; 
					\draw [red, thick] (0,2.5) to [out=180,in=70]  (-1.45, 0.1);
					\draw [red, thick] (1.25, 0.75) to [out=250,in=270] (-1.45, 0.1) ;
					\draw [black,fill] (-4.25, -0.5) circle [radius=0.05] node [black,below=.2] {$(1, 1, 1)$};
					\draw [thick] (-4.25, -0.5) to [out=10,in=180] (-3, 0) node[left=.5]{\small $\rot_j$}; 
				\end{tikzpicture}
				\caption{Existence of $\z$ intersecting the orbits of $\x$ and $\y$}
				\label{fig: im}
			\end{center}
		\end{figure}
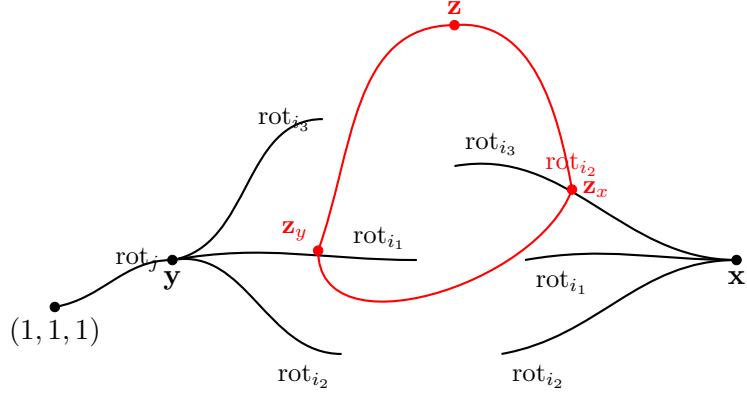

		\item[III.] If $\ord_p(\x)>p^{1/2}$ then by Proposition \ref{endgame} there exists a point $\x_C$ so that $\x_C$ is in the orbit of $\x$ under the rotation $\rot_{i_4}$ for some $i_4 \in \{1, 2, 3\}$. This implies that $\x$ is also in the orbit of $\x_C$ under the rotation $\rot_{i_4}$ and so we can write
		\[\x = \rot_{i_4}^{n_4}(\x_C).\]
		Then, we can connect $(1, 1, 1)$ to $\x_C$ as in Step II to get
		\[\x=\rot_{i_4}^{n_4} \cdots \rot_{i_1}^{n_1} \rot_j^n (1, 1, 1).\]
		
		\item[IV.] If $c<\ord_p(\x) \leq p^{1/2}$ then by Proposition \ref{middlegame} there exists a point $\x_M$ with $\ord_p(\x_M)>p^{1/2}$ where 
		\[\x_M = \rot_{i_5}^{m_5} \cdots \rot_{i_{t+4}}^{m_{t+4}} \x,\]
		and $t < \tau(p^2-1)$. Note that if $\x_i = \rot_i^m \x_j$ then this means that $\x_i$ is in the orbit of $\x_j$ under $\rot_i$. So as discussed in Step III, $\x_j$ is also in the orbit of $\x_i$ and we can write $\x_j=\rot_i^n \x_i$ for some integer $n$. Repeating this process gives us the following path from $\x_M$ to $\x$
		\[\x = \rot_{i_{t+4}}^{n_{t+4}} \cdots \rot_{i_5}^{n_5} \x_M.\]
		Since $\rot_p(\x_M)>p^{1/2}$ then we can connect $(1, 1, 1)$ to $\x_M$ as in Step III to get
		$\x = \rot_{i_{t+4}}^{n_{t+4}} \cdots \rot_{i_1}^{n_1} \rot_j^n (1, 1, 1).$
	\end{enumerate}
\end{algorithm}

\section{Proof of Theorem \ref{thm1}} \label{sec:thm1} We are now prepared to prove our first main result. 

\begin{proof}[Proof of Theorem \ref{thm1}] Take $\x \in \X^*(p)$. If $\ord_p(\x) > p^{1/2}$ given our assumptions and by Step III of Algorithm \ref{BGSalg} we can write
\[\x=\rot_{i_4}^{n_4} \cdots \rot_{i_1}^{n_1} \rot_1^n(1, 1, 1).\]
where $n_i \geq 0$, $0 \leq n \leq 5$ and the $i_j$ are not necessarily distinct. So, by Proposition \ref{sizes} we have
	\[\size(\x) < (3\epsilon)^{2^4 (n+1)(n_1+1)(n_2+1)(n_3+1)(n_4+1)},\]
	where $\epsilon=(3+\sqrt{5})/2$. Note that the $n_i$ correspond to path lengths from steps of the BGS algorithm. Since the proofs given in \cite{BGS} are nonconstructive, we upper bound each $n_{i}$ by the corresponding rotation order. By assumption we have $n \geq 5$ and by Lemma \ref{ord} we have $n_i \leq 2p$ for each $i \in \{1, 2, 3, 4\}$, and so the result follows.
\end{proof}

Observe that, using Step IV of Algorithm \ref{BGSalg} and Proposition \ref{sizes}, we can obtain the following upper bound for minimal lifts of points $\x \in \X^*(p)$ in the middlegame. 

\begin{proposition} \label{middlegamebound} Let $p$ be a prime so that $\ord_p(\rot_1^n(1, 1, 1))>p$ for $0 \leq n \leq 5$ and suppose that $\x \in \X^*(p)$ with $\ord_p(\x)>c$, where $c$ is the absolute constant introduced in the Middlegame of \cite{BGS}. Let $\tilde{\x}$ be a lift of $\x$ of minimal size. Then,
\[\size(\tilde{\x}) < (3\epsilon)^{96(2p+1)^{4 +t/2}},\]
where $t=\tau(p^2-1)$ denotes the number of positive divisors of $p^2-1$.
\end{proposition}

The proof of Proposition \ref{middlegamebound} follows identically to the proof of Theorem \ref{thm1} above, where we additionally note that in the Middlegame our points have order upper bounded by $p^{1/2}$. A natural next step would be to explicitly compute the absolute constant $c$ references in \cite{BGS} in order to extend Theorem \ref{thm1} to a larger family of points in $\X^*(p)$

\subsection{Percentage of points in the cage} In Figure \ref{fig:cage}, the percentage of total points in $\X^*(p)$ which satisfy the conditions of Theorem \ref{thm1} is plotted for primes $p < 300$. Observe that this percentage appears to hover around 80\%. Note here that this percentage also includes parabolic points of order $p$, since they are a similar distance from $(1, 1, 1)$ as points in the cage. 

\begin{figure}[H]
    \includegraphics[scale=0.5]{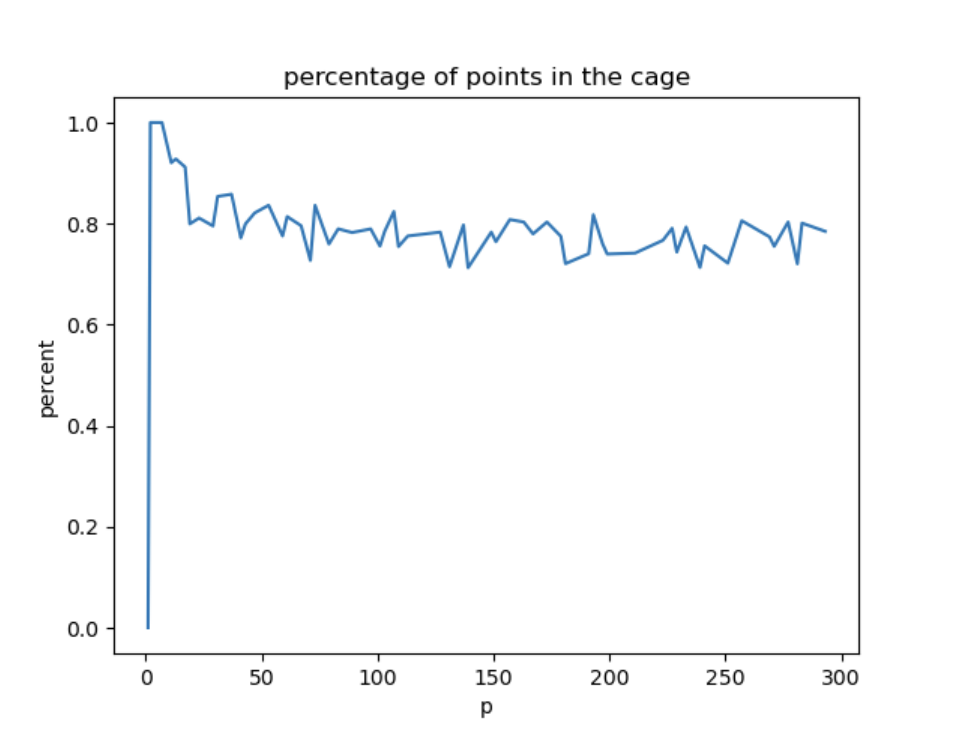}
    \label{fig:cage}
    \caption{Percentage of points from $\X^*(p)$ in $\C(p)$}
\end{figure}

As shown in the proof of Proposition \ref{ord}, note that a triple $\x$ is in the cage $\C(p)$ if and only if one of its coordinates $x$ satisfies any of the following
		\begin{enumerate}
		\item $x=-2/3$ if $x$ is parabolic. 
		\item $\epsilon_x$ generates $\F_p^\times$ if $x$ is hyperbolic, or
		\item $\epsilon_x=\rho_x^{p-1}$ where $\rho_x$ generates $\F_{p^2}^\times$ if $x$ is elliptic. 
	\end{enumerate}

Suppose that $x$ is a randomly chosen element in $\F_p$. If we assume the $\epsilon_x$ are equally likely to take the form in (2) or (3) as a randomly chosen element of $\F^\times$ and that generators are equally distributed in $\F^\times$, then the probability that $x$ satisfies conditions (1), (2) or (3) is given by

\[\frac{1}{p}+\frac{\phi(p-1)}{p-1} + \frac{\phi(p^2-1)}{(p-1)(p^2-1)}.\]

Since the probability that a Markoff $\mod p$ point $\x=(x_1, x_2, x_3)$ is in the cage is lower bounded by the probability that $x=x_1$ satisfies one of $(1)$, $(2)$ or $(3)$, an interesting future direction would be to use the argument above to derive a heuristic lower bound on the percentage of points in the cage for primes $p$ where $\phi(p-1)/p$ has a known asymptotic formula.

\section{Shortest Paths and Proof of Theorem \ref{thm2}} \label{sec:thm2}

In Proposition \ref{sizes} we saw that the size of our lift grows much faster when the corresponding path contains many switches between different rotations. By following the BGS algorithm we were able to make a small number of these switches, but had to compromise by assuming long paths as we traveled along the orbit of a single rotation due to the non explicit methods used in \cite{BGS}. In this section, we study lifts obtained from paths of shortest possible length with possibly many switches between rotations. The bound in Theorem \ref{thm2} are uniform if the following conjecture holds, which is expected to be true (see \cite{experiments}, for example). 

\begin{conjecture}[Super Strong Approximation] \label{SSA} For any prime $p$, the collection of Markoff mod $p$ graphs $(\G_p)_p$ forms an expander family. 
\end{conjecture}

An expander family is a collection of graphs $(G_i)_i$ that is ``highly connected but relatively sparse". We refer to \cite{kowalski} for the formal definition. To obtain a uniform bound, we will only need to know that the expansion constant of any expander family is bounded. As a Corollary to the following Proposition, we obtain a bound on the diameter of any Markoff $\mod p$ graph.

\begin{proposition}[Proposition 3.1.5 of \cite{kowalski}] Let $G$ be any finite non-empty connected graph. We have
	\[\diam(G) \leq 2 \frac{\log \frac{|G|}{2}}{\log\left(1+\frac{h(G)}{v}\right)} +3,\]
	where $v$ is the maximum number of edges at each vertex and $h(G)$ is the expansion constant of $G$.  
\end{proposition}

Note that $|\X^*(p)| = p^2 \pm 3n$, depending on whether $p \equiv 1 (\mod 4)$. 
Since the $v=3$ in $\G_p$ for any prime $p$ we have the following Corollary.

\begin{corollary} \label{diam} If Conjecture \ref{SSA} holds, then
	\[\diam(\G_p) \leq C \log\left(\frac{p^2+3}{2}\right).\]
	where $\G_p$ is the Markoff $\mod p$ graph and $C$ is a constant given by
	\[C = \frac{5}{\log\left(1+\frac{h}{3}\right)},\]
	where $h$ is an upper bound for the expansion constant of $(G_p)_p$. 
\end{corollary}

We our now prepared to prove our second main result. 

\begin{proof}[Proof of Theorem \ref{thm2}] Let $\x \in \X^*(p)$ and suppose that the shortest path from $(1, 1, 1)$ to $\x$ in $\G_p$ has length $\ell$. Then $\x$ is of the form
	\[\x=\rot_{i_s}^{n_s} \cdots \rot_{i_1}^{n_1} (1, 1, 1),\]
	where $1 \leq s \leq \ell$ and $n_1 + \cdots + n_s=\ell$. So by Proposition \ref{sizes}, if $\tilde{\x}$ is a lift of minimal size of $\x$ we have $\size(\tilde{\x}) < (3\epsilon)^\alpha$ where
	\begin{align*}
		\alpha &= 2^{s-1}(n_1+1)\cdots (n_s+1) \\
		&< 2^{s-1} 5^{\ell}\\
		& < 2^{4\ell}
	\end{align*}
	where the second to last inequality follows from Lemma \ref{partitionprod} below. 
	By Corollary \ref{diam} we have 
	\[\alpha < 2^{4C \log((p^3+3)/2)} 
	< \left(\frac{p^3+3}{2}\right)^{4C},\]
	where \[C=\frac{5}{\log(1+\frac{h}{3})}\]
	as in Lemma \ref{partitionprod}.
\end{proof}

\begin{lemma} \label{partitionprod} For any positive integer $\ell$ we have
	\[\max\{(n_1+1)\cdots (n_s+1) \mid (n_1, \dots, n_s) \text{ partitions } \ell \text{ and } 1 \leq s \leq n \} \leq 5^\ell.\]
\end{lemma}

\begin{proof}
	Suppose that $(n_1, \dots, n_s)$ partition $\ell$, and that one of our parts is larger than 4, say $n_1>4$. Then $n_1 < 3(n_1-3)$ and we get
	\[(n_1+1)(n_2+1)\cdots(n_s+1) < (3 (n_1-3)+1)(n_2+1) \cdots (n_s+1).\]
	Since $(3, n_1-3, n_2, \dots, n_s)$ is also a partition $\ell$ then $(n_1+1) \cdots (n_\ell+1)$ is not maximal. So, any partition $(n_1, \dots, n_s)$ maximizing $(m_1+1)\cdots (n_s+1)$ must have $n_i \leq 4$ for every $i=1, \dots, s$. 
\end{proof}

\subsection{Data to support improvements on average} Note that the bound in Theorem \ref{thm2} assumed our shortest path from $\x$ in $\G_p$ to $(1, 1, 1)$ switches between distinct rotations maximally. By Proposition \ref{sizes}, this contributes doubly exponentially to the growth of the corresponding lift. From numerical experimentation, we expect that this bound can be improved if we consider lifts on average. \\

The histogram in Figure \ref{fig:distribution}  plots the frequency of the logarithmic size of Markoff triples of fixed length $\ell=14$ from $(1, 1, 1)$ in the Markoff tree defined by the rotations $\rot_i$; that is, the tree with root node $(1, 1, 1)$ and edges defined by $\{\rot_1, \rot_2, \rot_3\}$.

\begin{figure}[H]
\includegraphics[scale=0.5]{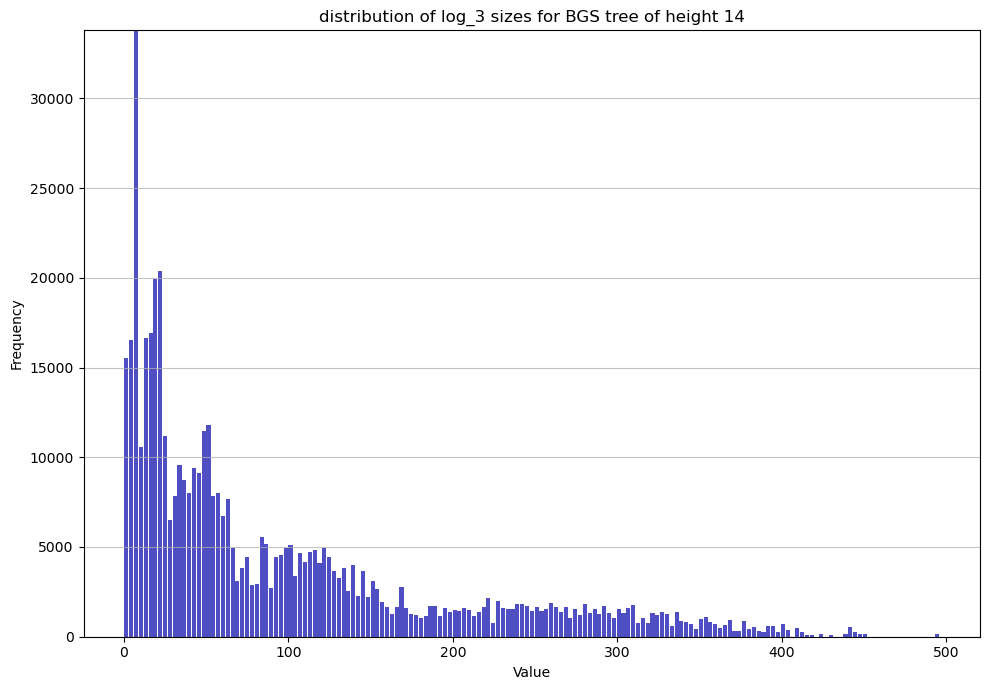} 
				\caption{Distribution of logarithmic sizes at level 14}
                \label{fig:distribution}
\end{figure}

\vspace{1em}
This data suggests that sizes are distributed in such a way that favor smaller lifts, and so we expect that the largest sizes used in our upper bound occur infrequently enough to allow for improvement to this bound on average. An interesting future direction would be an explicit calculation of this distribution in order to obtain an improved upper bound on the size of lifts on average using the methods in Theorem \ref{thm2}.

\section{Short Paths Along Parabolic Orbits} \label{sec:parabolic} Suppose  $\x=(x_1, x_2, x_3)$ is in $\X^*(p)$ with $x_i$ parabolic. For convenience, suppose that $i=1$ and set $x=x_1$. In Section \ref{parabolic} we showed that when $p \equiv 1 (\mod 4)$ and $x=2/3$ the conic section consists of the disjoint orbits
\[C_i(\x)=M_{\x^{(+)}} \cup M_{\x^{(-)}}\]
where $M_{\x^{(+)}}$ and $M_{\x^{(-)}}$ are defined in Equations (\ref{eqx-}) and (\ref{eqx+}). By Proposition \ref{ord} and Lemma \ref{conic}, when $x=-2/3$ we have $C_i(\x)=M_{\x, i}$
and using a similar analysis to that in Section \ref{parabolic} we get
\[M_{\x, i} = \left\{\left(\frac{-2}{3}, -1 \pm \frac{2}{3} in, 1 \mp \frac{2}{3} i (n+1)\right)\right\}.\]
Now, let $p$ be a prime so that there exists a point $\y=(y_1, y_2, y_3)$ in the cage with
\[\rot_1^n (1, 1, 1)=\y\]
for $0 \leq n \leq 5$. If it's the case that the maximal index of $\y$ is not equal to one, then we can connect any parabolic point directly to $\y$ as discussed in Section \ref{parabolic} by a path of the form 
\[\x=\rot_i^{n_i}\rot_j^{n_j}\rot_1^n(1, 1, 1).\]
As in the proof of Theorem \ref{thm1} we can take $n_i, n_j$ smaller than the largest rotation order, using Proposition \ref{sizes} gives the following.

\begin{proposition} \label{parabolicbound} Let $\x \in \X^*(p)$ be of the form $\x=(\pm 2/3, x_2, x_3)$. Suppose that the point $\y$ above has maximal index $i \in \{2, 3\}$ and let $\tilde{\x}$ be a lift of $\x$ of minimal size. Then 
	\[\size(\tilde{x}) \leq (3\epsilon)^{20(2p+1)^2}.\]
\end{proposition}

Since parabolic points have large orbits, it's expected that many points are ``close" to parabolic orbits. Because the bound in Proposition \ref{parabolicbound} is better than our previously obtained bounds, a natural next direction would be to investigate points that have short paths to nearby parabolic orbits.


\bibliography{bib}
\bibliographystyle{plain}

\end{document}